\documentclass[12pt]{amsart}
\usepackage[latin1]{inputenc}
\usepackage{color}
\usepackage{enumerate}
\usepackage{amssymb}
\usepackage[all]{xy}
\usepackage{hyperref}

\newtheorem{thm}{Theorem}[section]

\newtheorem{lem}[thm]{Lemma}

\theoremstyle{definition}
\newtheorem{defn}[thm]{Definition}
\theoremstyle{remark}

\newtheorem{prob}{Problem}

\newcommand{\To}{\longrightarrow}

\usepackage{color}
\usepackage{enumerate}
\begin{document}

\title{On fragmentable compact lines}
\author[A.\ Aviles]{Antonio Avil\'es}
\address{Universidad de Murcia, Departamento de Matem\'{a}ticas, Campus de Espinardo 30100 Murcia, Spain.}
\email{avileslo@um.es}

\author[G.\ Mart\'{i}nez-Cervantes]{Gonzalo Mart\'{i}nez-Cervantes}
\address{Universidad de Murcia, Departamento de Matem\'{a}ticas, Campus de Espinardo 30100 Murcia, Spain.}
\email{gonzalomartinezcervantes@gmail.com, gonzalo.martinez2@um.es}

\author[G.\ Plebanek]{Grzegorz Plebanek}
\address{Uniwersytet Wroc{\l}awski, Instytut Matematyczny, Pl. Grunwaldzki 2/4, 50-384 Wroclaw, Poland.}
\email{grzes@math.uni.wroc.pl}
\author{Stevo Todorcevic}
\address{Department of Mathematics, University of Toronto, Toronto, Canada, M5S 3G3.
	Institut de Math\'{e}matiques de Jussieu, CNRS UMR 7586, Case 247, 4 place Jussieu, 75252 Paris Cedex, France }%
\email{stevo@math.jussieu.fr, stevo@math.toronto.edu}

\thanks{The first and second authors are supported by projects MTM2014-54182-P and MTM2017-86182-P (MINECO,AEI/FEDER, UE) and by project 19275/PI/14 (Fundaci\'on S\'eneca). The fourth author is partially supported by Grants from NSERC (455916) and CNRS (IMJ-PRG UMR7586).}

\keywords{fragmentability; Radon-Nikod\'{y}m compact; compact line}

\subjclass[2010]{46B26,06A5,54F05}

\begin{abstract}
We prove that if a compact line is fragmentable, then it is a Radon-Nikod\'{y}m compact space.
\end{abstract}

\maketitle

\section{Introduction}
Let $X$ be a topological space and let $d:X\times X\To [0,\infty)$ be a metric on $X$. We do not assume that the topology on $X$ coincides with the topology induced by $d$. In fact, we consider the following two weaker conditions:

\begin{itemize}
	\item We say that $d$ fragments $X$ if for every nonempty subset $Y\subset X$ and for every $\varepsilon>0$ there exists a nonempty relatively open $V\subset Y$ of $d$-diameter less than $\varepsilon$, that is
	$$\sup\{d(x,y) : x,y\in V\}<\varepsilon.$$
	
	\item We say that $d$ is lower semicontinuous if for every $r>0$, the set $$\{(x,y)\in X\times X : d(x,y)\leq r\}$$ is closed.
\end{itemize}

A compact space $K$ is {\em fragmentable} if there exists a metric that fragments it, while it is called a {\em Radon-Nikod\'{y}m (RN) compact}
if there exists a lower semicontinuous metric that fragments it. These classes of compacta  arise in the theory of nonseparable Banach spaces, where they play a prominent role in problems concerning differentiability, renorming and others,  cf.\ \cite{FabianWA,Namiokasurvey}. Answering a problem by Namioka \cite{Namioka}, 
Orihuela, Schachermayer, Valdivia \cite{OSV} found the first examples of fragmentable compact spaces that are not Radon-Nikod\'{y}m. These were all the Gul'ko non-Eberlein compact spaces, like those constructed by Talagrand in \cite{Ta1,Ta2}.

By a {\em compact line} we mean a linearly ordered set that is a compact space in the topology generated by the base of  open intervals.
Fragmentability of compact lines was considered  in \cite{RNorder}, where it was asked whether a fragmentable compact line must be Radon-Nikod\'{y}m compact.
A positive solution to this problem appeared in a preprint  by Smith \cite{arxivSmith}. However, that proof happened to contain a gap, a fact communicated to us by its author, and for that reason the preprint was never published. In this note, we will give a proof that indeed every fragmentable compact line is RN. Our argument makes use of a characterization of $\sigma$-scattered orders due to  Ishiu and Moore \cite{ishiuline}.

\section{Preliminaries}

We first recall some basic facts about stationary sets, see Definition 8.21 and onwards in Jech's book ~\cite{jech}. We write $[K]^\omega$ for the family of all countable subsets of a set $K$. A family $\Delta\subset [K]^\omega$ is said to be a \emph{club} if

\begin{enumerate}[(i)]
	\item $\Delta$ is \emph{closed},  that is, for every  increasing sequence $M_1\subset M_2\subset\ldots$  of elements of $\Delta$ we have $\bigcup_n M_n\in \Delta$;
	\item $\Delta$ is \emph{unbounded},  that is, for every $M\in [K]^\omega$ there exists $N\in\Delta$ such that $M \subseteq N$.
\end{enumerate}

The family $\{M\in [K]^\omega : M_0\subset M\}$, for a fixed countable $M_0\subset K$, is an obvious example of a club.
The countable intersection of clubs is a club.

A set $S\subset [K]^\omega$ is {\em stationary}  if it interesects every club. Thus, a set is said to be non-stationary if it is disjoint from some club. Countable union of non-stationary sets is non-stationary.  The \emph{pressing-down} or \emph{Fodor's} lemma in this setting
asserts that if  $S\subset [K]^\omega$ is stationary and $f:S\To K$ satisfies $f(M)\in M$ for all $M\in S$, then there is a stationary subset $S'\subset S$ where $f$ is constant,
see \cite[Theorem 8.24]{jech}.

Let us now fix a compact line $K$, and let $L$ be the set all elements of $K$ that are either left-isolated or right-isolated in $K$, that is
\begin{equation}\label{el}
 L = \left\{x\in K : \left( \exists y \in K\setminus\{x\}\right)   [\min(x,y),\max(x,y)] = \{x,y\} \right\}.
 \end{equation}

For the rest of this section we suppose  that $K$ is a zero-dimensional compact line.  Then for every $x<y$ in $K$ there exist $u,v\in L$ such that $[u,v]=\{u,v\}$ and $x\leq u < v \leq y$;
otherwise, there would be an element between any $u,v\in [x,y]$ and the interval $[x,y]$ would be connected.
Consequently, every element $x\in K\setminus L$  can be expressed as a supremum of a subset of $L$ that does not have a supremum in $L$, and therefore $K$ is the completion $\widehat{L}$ of $L$ in the sense of \cite{ishiuline}.

 A linear order $Z$ is said to be {\em scattered} if it does not contain any copy of the usual linear order on $\mathbb{Q}$.
 A {\em $\sigma$-scattered} order is one that can be written as a countable union of scattered suborders.


We shall use the following concept introduced in \cite{ishiuline} (here $L$ is defined by (\ref{el})).

\begin{defn}
Given $x\in K$ and a set $M\subset K$, we say that $M$ \emph{captures} $x$ if there exists $z\in M$ such that only finitely many elements of $M\cap L$ lie between $x$ and $z$.
\end{defn}

Obviously, this is equivalent to saying that there exists $z\in M$ such that no elements of $M\cap L$ lie strictly between $x$ and $z$. Thus, the phrase 
$M$ \emph{does not capture $x$} means that for every $z\in M$ there is $y\in M\cap L$  strictly between $x$ and $z$. It is convenient to rephrase this in a different way. Given $M\subset K$ and $x\in K$, consider
\begin{equation}\label{x0}
x_M^0 = \sup\{y\in L\cap M : y < x\} \in K,
\end{equation}
\begin{equation}\label{x1}
x_M^1 = \inf\{y\in L\cap M : y > x\} \in K.
\end{equation}

\begin{lem}
 $M$ does not capture $x$ if and only if $[x^0_M,x^1_M]\cap M = \emptyset$.
\end{lem}

In the next section we use a result of Ishiu and More characterizing $\sigma$-scattered orders  $L$
in terms of stationary sets in $[K]^\omega$ consisting of countable subsets of $K$ that capture all the elements of $L$.

\section{The main result}

We first show that  a fragmentable compact line is RN compact provided it is zero-dimensional, and then
conclude the general case in Theorem \ref{general} below.

\begin{thm}\label{0dimensional}
	Let $K$ be a zero-dimensional compact line and let $L$ be defined by (\ref{el}). Then the  following are equivalent:
	\begin{enumerate}[(i)]
		\item $K$ is fragmentable;
		\item $\Gamma = \{M\in [K]^\omega : \exists\ x\in L  \text{ not captured by } M \}$ is non-stationary subset of $[K]^\omega$;
		\item $L$ is a $\sigma$-scattered order;
		\item $K$ is Radon-Nikod\'{y}m compact.
	\end{enumerate}
\end{thm}

\begin{proof}
$[(i)\Rightarrow (ii)]$ Suppose that $K$ is fragmented by a metric $d$ but  $\Gamma$ is stationary. For every $M\in\Gamma$ we choose an element $x_M = x\in L$ 
 that is not captured by $M$, and $x^0_M$ and $x^1_M$ the corresponding elements of $K$ defined by formulas (\ref{x0}) and (\ref{x1}). 
 Since $x^0_M$ and $x^1_M$ are not in $M$, to every $M\in \Gamma$ we can associate an increasing sequence $(v_M^i)\subset L\cap M$ whose supremum is $x^0_M$ and a decreasing sequence $(w_M^i)\subset L\cap M$ whose infimum is $x^1_M$. 
 
 Note that $x_M^0\neq x_M^1$ for every $M\in \Gamma$. Indeed, otherwise we would have $x^1_M=x^0_M = x_M\in L$ 
 so $x_M$ would be isolated neither from the left nor from the right,
 which is in contradiction with the formula (\ref{el})  defining $L$.
Consequently, $d(x^0_M,x^1_M)>0$ for every $M\in \Gamma$, and there must exist $n<\omega$ such that the set
  $$\Gamma_n = \{M\in \Gamma : d(x^0_M,x^1_M)>1/n\}$$ 
  is stationary.
\medskip

\noindent {\sc Claim.}  The set
$$A = \{M\in \Gamma_n : \exists\ i<\omega :
\{ N\in \Gamma_n : v^i_M <x^0_N < x^0_M\} \text{ is non-stationary}\}
\}$$

is non-stationary.
\medskip

{\em Proof of claim.} Suppose that, on the contrary, $A$ is stationary. 
Then, without loss of generality, we can assume that the fact that $M\in A$ is witnessed by the same  $i<\omega$.
 By the pressing down lemma we can find a stationary set $A_0\subset A$  such that 
$v^i_M = v$ for all $M\in A_0$. Let $\xi = \sup\{x^0_M : M\in A_0\}$. The set $A_1 = \{M\in A_0 : \xi\in M\}$ is stationary. If $M\in A_1$, then notice that $x^1_M<\xi$ (otherwise $\xi\in [x^0_M,x^1_M]\cap M = \emptyset$). Thus, for each $M\in A_1$ we can choose $w_M^i<\xi$. Using  the pressing down lemma again,
 we can find $w$  such that $A_2 = \{M\in A_1 : w^i_M = w\}$ is stationary. Since $\xi$ is the supremum, we can find $P\in A_0$ such that $w<x^0_{P}$. On the one hand, $\{N\in \Gamma_n : v<x^0_N<x^0_{P}\}$ is non-stationary since $P\in A$. On the other hand, that set contains $A_2$ which is stationary, a contradiction.
\medskip

Now, once the claim is proved, a similar argument gives that
$$B = \{M\in \Gamma_n : \exists\ i<\omega :
\{ N\in \Gamma_n : w^i_M > x^1_N > x^1_M\} \text{ is non-stationary}\}
\}$$
is non-stationary.

We now examine  the set $$P = \{x^0_M,x^1_M : M\in \Gamma_n\setminus(A\cup B)\}\subset K.$$

By fragmentability, this set should have a nonempty relative open subset of diameter less than $1/n$. 
Suppose $W\subset P$ is such an open set;   we can assume that $W$ is  a relative open interval, and that it contains a point of the form, say, $x^0_M$. 
Then, since $M\not\in A$ there are stationarily many $x^0_N$ in $W\cap P$ with $x^0_N < x^0_M$. There will be some of these $N$ such that $M\subset N$ and $x^0_M\in N$, because there are club many such $N$'s. Since $N$ does not capture $x_N$ and $x^0_M\in N$, we must have $x^1_N\leq x^0_M$. Thus $x^0_N,x^1_N\in W$, but $d(x^0_N,x^1_N)>1/n$ because $N\in \Gamma_n$, a contradiction.

\noindent $[(ii)\Rightarrow (iii)]$ This is the content of Theorem 2.11 of \cite{ishiuline}.

\noindent $[(iii) \Rightarrow (iv)]$ Let $L^-\subset L$ be the set of all right-isolated elements of $K$, that is
$$ L^- = \{s\in K : \exists s^+>s : [s,s^+] = \{s,s^+\} \}.$$
If $L$ is $\sigma$-scattered, then so is $L^-$, and we can  write $L^- = \bigcup_n S_n$ where each $S_n$ is a scattered ordered set. We can suppose that $S_1\subset S_2\subset \ldots$. Since $K$ is zero-dimensional, for every $x<y$ there is $s\in L$ with $x\leq s <y$. 

Define a function $d$ on $K\times K$ given for $x<y$ by the formula
$$d(x,y) =1/n, \mbox{ where }  n=\min\{k : \exists s\in S_k : x\leq s < y \} .$$
Then $d$ is easily seen to define  a metric, in fact $d(x,y) \leq \max\{d(x,z),d(z,y)\}$ for every $x,y,z\in K$. 
Let us check that $d$ is lower semicontinuous. For this it is enough to notice that if we  take $x<y$ such that $d(x,y)=1/n > r$, and $s\in S_n$ such that $x\leq s < y$,  
then the product of half-lines
$$(-\infty,s^+)\times (s,\infty ),$$ 
is a neighbourhood of $(x,y)\in K\times K$ in which   all the pairs are at distance greater than $r$. 

Now it remains to check that the metric $d$ fragments $K$. Suppose, on the contrary, 
that there is a nonempty subset $Y\subset K$ all of whose nonempty relative open subsets have $d$-diameter at least $1/n$.
 For every finite sequence of 0's and 1's $\sigma\in 2^{<\omega}$, we define, by induction on the length of $\sigma$, a nonempty relative open interval $Z_\sigma\subset Y$ and $s_\sigma\in S_n$ as follows.
 
 We start with $Z_\emptyset = Y$. Given $Z_\sigma$, we find two points $x<y$ in $Z_\sigma$ with $d(x,y)\geq 1/n$, we pick $s_\sigma\in S_n$ such that $x\leq s_\sigma < y$, and we define
$$Z_{\sigma^\frown 0} = Z_\sigma\cap (-\infty,s_\sigma^+), \quad  Z_{\sigma^\frown 1} = Z_\sigma\cap (s_\sigma,\infty).$$ 
Then $Z_{\sigma^\frown i}$ are again nonempty open relative subintervals of $Y$.

 Once the construction is done, it gives the set $\{s_\sigma : \sigma\in 2^{<\omega}\}\subset S_n$, which is  a countable ordered set such that between every two points there is a third one. The classical theorem of Cantor implies that this set is order-isomorphic to $\mathbb{Q}$, which contradicts that $S_n$ is a scattered order.

Finally, the implication $(iv)\Rightarrow (i)$  is evident from the definitions, and the proof is complete.
\end{proof}

\begin{thm}\label{general}
	If $K$ is a fragmentable compact line, then $K$ is Radon-Nikod\'{y}m compact.
\end{thm}

\begin{proof}
 Let $d$ be a metric that fragments $K$. For every $n$ find a maximal family $\mathcal{F}_n$ of pairwise disjoint closed infinite intervals of $K$ of $d$-diameter less than $\frac{1}{n}$, and let 
 $$K_n = K\setminus \bigcup\{ (a,b) : [a,b]\in \mathcal{F}_n\}.$$ 
 The set $K_n$ is clearly compact; let us check that it is also zero-dimensional.

 Set $x,y\in K_n$ with $x<y$.  If $(x,y)$ contains a finite open interval, we are done. Otherwise, by fragmentability, it contains an infinite open subinterval of diameter less than $1/n$, which contains a further infinite closed subinterval of diameter less than $1/n$. The maximality of $\mathcal{F}_n$ provides an interval $[a,b]\in \mathcal{F}_n$ with $[a,b]\cap (x,y) \neq \emptyset$. Since $x,y\in K_n$, we get that $x,y \notin (a,b)$. This implies that $x\leq a<b \leq y$, which in turn proves that $K_n$ is zero-dimensional.

 It follows from  Theorem~\ref{0dimensional} that  the compact space $K_n$ is RN. For every $[a,b]\in \mathcal{F}_n$ fix a continuous nondecreasing function $f_{(a,b)}:K\To [0,1]$ such that $f_{(a,b)}(a) = 0$ and $f_{(a,b)}(b)=1$. Let $P(K_n)$ be the compact space of probability measures on $K_n$ endowed with the weak$^\ast$ topology, that is  the coarsest topology that makes  the functional $\mu\mapsto \int h\; {\rm  d} \mu$ continuous,
 for every continuous function $h:K_n\To\mathbb{R}$. Consider the function $\Phi_n:K\To P(K_n)$ given by 
 $$\Phi_n(x)=
 \begin{cases} 
 \delta_x &   \text{ if } x\in K_n, \\
 \Phi_n(x) = (1-f_{(a,b)}(x))\delta_a + f_{(a,b)}(x)\delta_b &  \text{ if }x\in [a,b]\in\mathcal{F}_n.\end{cases}$$
 Observe that $\Phi_n:K\To P(K_n)$ is weak$^\ast$ continuous: take any continuous function $h:K_n\To\mathbb{R}$.
  Then 
  $$\int h\ d\Phi_n(x) = \begin{cases} 
 h(x) &\text{ if } x\in K_n, \\ 
 (1-f_{(a,b)}(x))h(a) + f_{(a,b)}(x)h(b) & \text{ if } x\in [a,b]\in\mathcal{F}_n,\end{cases}$$
 which plainly shows that  the mapping $x\mapsto \int h\; {\rm  d}\Phi_n(x)$ is continuous.
 
  Consider now the diagonal mapping
 $$\Phi:K\To \prod_n P(K_n), \quad  \Phi(x) = (\Phi_1(x),\Phi_2(x),\ldots).$$ 
 Then $\Phi$ is continuous and it is, moreover,  injective: 
 observe  for any $x,y\in K$ with $d(x,y)\ge 1/n$ we have $\Phi_n(x)\neq\Phi_n(y)$, so the mappings $\Phi_n$ separate points of $K$.

We conclude that $K$ is homeomorphic to a closed subspace of $\prod_n P(K_n)$. The class of RN compacta is closed under countable products and under the operation of 
taking the space of probability measures, see  \cite{Namioka}. Hence,  $K$ is RN compact, and the proof is complete.
\end{proof}

\section{Final remarks}

One may wonder whether Theorem~\ref{general} can be generalized to a class of compact spaces larger than the class of linearly ordered compact spaces. Since the class of RN compact spaces is closed under taking subspaces and under countable products, it can be easily checked that Theorem~\ref{general} also applies to countable products of compact lines.

A natural class which generalizes both the class of linearly ordered compact spaces and the class of RN compact spaces is the class of weakly Radon-Nikod\'ym (WRN) compacta (see \cite{Gonzalo2} for the definition). Nevertheless, S. Argyros proved the existence of Gul'ko WRN compact spaces which are not Eberlein (see \cite{Argyros} and \cite[Section 2.4]{Gonzalo}). As we highlighted in the Introduction, every Gul'ko non-Eberlein compact space is an example of a fragmentable compact space which is not Radon-Nikod\'ym.
Therefore, we cannot change compact lines by WRN or Gul'ko compacta in Theorem~\ref{general}.

Notice that the class of fragmentable compact spaces is stable under continuous images \cite{Ribarska}. Another remarkable example of a fragmentable compact space which is not RN is the one constructed in	\cite{RNimage}, which indeed is a continuous image of a RN compact space. Nevertheless, one might expect that Theorem~\ref{general} can be extended to continuous images of compact lines:

\begin{prob}
Is every continuous image of a compact line Radon-Nikod\'ym provided it is fragmentable?
\end{prob}

Recall that the class of continuous images of compact lines coincides with the class of monotonically normal compact spaces by a result of M.E. Rudin \cite{Rudin}.

On the other hand, every zero-dimensional compact line can be obtained from a metric compact space with inverse limit systems consisting of simple extensions.
In the Boolean algebra setting, these extesions are called minimal extensions and the Boolean algebras obtained by this process are known as minimally generated Boolean algebras. We do not know the answer to the following problem:

\begin{prob}
Is the Stone space associated to a minimally generated Boolean algebra RN provided it is fragmentable?
\end{prob}

\end{document}